\documentclass[12pt]{amsart}
\usepackage{amsfonts}
\usepackage{amsmath, amsthm, amssymb,color}
\usepackage{mathrsfs}
\usepackage{latexsym}
\usepackage{txfonts}
\usepackage{bm}

\numberwithin{equation}{section}
\allowdisplaybreaks

\newtheorem{lemma}{Lemma}[section]

\newtheorem{theorem}[lemma]{Theorem}

\newtheorem{proposition}[lemma]{Proposition}
\newtheorem{definition}[lemma]{Definition}
\newtheorem{corollary}[lemma]{Corollary}
\newtheorem{example}[lemma]{Example}
\newtheorem{exercise}[lemma]{Exercise}
\theoremstyle{definition}
\newtheorem{remark}[lemma]{Remark}

\usepackage{graphics}
\usepackage{graphicx}
\usepackage{epsfig}

\newcommand{\bth}{\begin{theorem}}
\newcommand{\ethe}{\end{theorem}}

\newcommand{\bre}{\begin{remark}\em }
\newcommand{\ere}{\end{remark}}

\newcommand{\ble}{\begin{lemma}}
\newcommand{\ele}{\end{lemma}}

\newcommand{\bde}{\begin{definition}}
\newcommand{\ede}{\end{definition}}
\newcommand{\bco}{\begin{corollary}}
\newcommand{\eco}{\end{corollary}}

\newcommand{\bpr}{\begin{proposition}}
\newcommand{\epr}{\end{proposition}}

\newcommand{\bexer}{\begin{exercise}}
\newcommand{\eexer}{\end{exercise}}

\newcommand{\bexam}{\begin{example}\rm }
\newcommand{\eexam}{\end{example}}

\newcommand{\bali}{\begin{align}}
\newcommand{\eali}{\end{align}}

\newcommand\N{\mathbb N}
\newcommand\Z{\mathbb Z}

\newcommand\R{\mathbb R}

\begin{document}

\bibliographystyle{alpha}
\title[Log-convexity and cycle index polynomials 
]{Log-convexity and the cycle index polynomials with relation to 
compound Poisson distributions}
\today
\author[M. Matsui]{Muneya Matsui}
\address{Department of Business Administration, Nanzan University,
18 Yamazato-cho Showa-ku Nagoya, 466-8673, Japan}
\email{mmuneya@nanzan-u.ac.jp}

\begin{abstract} 
 We extend the exponential formula by Bender and Canfield (1996), 
 which relates log-concavity and the cycle
 index polynomials. The extension clarifies the log-convexity
 relation. The proof is by noticing the property of a compound Poisson
 distribution together with its moment generating function. We also give
 a combinatorial proof of extended ``log-convex part'' referring to Bender and
 Canfield's approach, where the formula by Bruijn and Erd\"os (1953) is
 additionally exploited. The combinatorial approach yields richer
 structural results more than log-convexity. Furthermore, we consider
 normal and binomial convolutions of sequences which satisfy the
 exponential formula. The operations generate interesting examples which
 are not included in well known laws about log-concavity/convexity. 

\end{abstract}

\keywords{Cycle index polynomials, compound Poisson, symmetric group,
 Infinitely divisible, generating function. log-convexity.} 
\subjclass[2010]{05A20 ; 60E07 : 60E10 ; 60C05}
\thanks{
Muneya Matsui's research is partly supported by the JSPS Grant-in-Aid
for Young Scientists B (16K16023).
}
\maketitle 
\section{Introduction}

The focus is to study properties of non-negative sequences $(a_k)_{k\ge
0}$ and $(b_k)_{k\ge 0}$ from those of $(c_k)_{k\ge 1}$ such that they
are linked by 
\begin{align}
\label{eq:main-relation}
 \sum_{k=0}^\infty a_k u^k = \sum_{k=0}^\infty
 \frac{b_ku^k}{k!}=\exp\Big(
 \sum_{j=1}^\infty \frac{c_ju^j}{j}
\Big).
\end{align}
The sequence $(a_k)$ is regarded as the cycle index polynomials of
symmetric group \cite{polya:1937,redfield:1927}. 
From the relation
\eqref{eq:main-relation}
Bender and Canfield 
have shown the log-concavity and almost log-convexity of $(a_k)$, 
assuming that $(c_k)_{k\ge 0}$ with $c_0=1$ is log-concave \cite[Theorem
1]{bender:canfield:1996}. 

 The first main result of this paper (Section \ref{main}) is an extension of \cite[Theorem
 1]{bender:canfield:1996}, namely, we show the log-convexity of $(a_k)$ from that
 of $(c_k)$. Our approach is to notice the probabilistic
 interpretation of \eqref{eq:main-relation} together with distributional
 properties of compound Poisson (CP for short) distributions. This approach leads to  
 an alternative proof for the previous log-concavity result
 \cite[Theorem1]{bender:canfield:1996} which was originally derived from
 combinatorial study. 
 In the same section, we specify (probably
 unknown) common characteristics between these combinatorial and probabilistic
 approaches toward \cite[Theorem1]{bender:canfield:1996}. Inspired by
 this specification, we consider combinatorial proof
 for the extended ``log-convex part''. This combinatorial proof yields
 richer results more than log-convexity, which are related to the structure
 of the cycle index and which were analyzed as \cite[Theorem
 2]{bender:canfield:1996} in the log-concave case. 

 In Section \ref{sec:convo}, we consider log-convex/concave properties of normal
 convolution of $(a_k)$ and binomial convolution of $(b_k)$ which
 satisfy the exponential formula \eqref{eq:main-relation}. Both convolutions are obtained by
multiplying the r.h.s. of \eqref{eq:main-relation}. The operations
 generate interesting sequences which are not included in known
 principles for log-concavity/convexity. Our example sequences are
 mostly by those of probability mass functions.

In the rest, we state related literature of both combinatorics and probability
and statistics. This is worth describing since they developed the
similar theories independently without having proper 
intersections.

In combinatorics the log-concavity/convexity of combinatorial  
sequences 
has been intensively studied (see Stanley \cite{stanley:1989} and Brenti \cite{brenti:1989}
for log-concavity, and see Lin and Wang \cite{lin:wan:2007} for log-convexity and references
therein). Especially log-concavity is closely related with
unimodality and they 
has been studied together.  
Bender and Canfield's result \cite{bender:canfield:1996} serves as a
tool to judge log-concavity. The method is to investigate the property
of the original
sequence 
by its exponential generating function (GF for abbreviation). 
Indeed the relation \eqref{eq:main-relation} corresponds to the GF for the cycle
index of symmetric groups (see Remark \ref{rem:relation} (i) bellow). The
method by GFs is a powerful tool for solving combinatorial problems (see \cite{wilf:1994}).

The properties are also significant in probability and statistics. 
To confine the related topics, 
they play a crucial role in  
the class of infinitely divisible (ID for short) distributions, one of
most important probability distributions in both theory and applications. The log-concavity
is a useful tool for investigating the unimodality of ID distributions
(see e.g. Sato
\cite{sato:1999} and Steutel \cite{steutel:1970}). Indeed, the class of 
strong unimodal probability density/masss functions is 
equivalent to that of log-concavity density/masss functions. The log-convexity
characterizes ID distributions on $\Z_+$, i.e. it gives a sufficient
condition for distributions to be ID \cite[Theorem 51.3]{sato:1999}. A
sufficient conditions for log-concavity/convexity of
$\Z_+$ valued ID distributions 
is given by \cite{hansen:1988}. For ID distributions since explicit
expressions of most density/masss functions are unavailable, the 
most useful tools are their characteristic functions
(\cite[p.7]{sato:1999}) which are substantially equivalent to
GFs. 

Therefore quite similar problems are investigated with similar/different methods
in these two fields, and
connections are sometimes discovered and mentioned : e.g. $n$th Bell number
corresponds to $n$th moment of the CP with mean $1$. 
But they are sparse and not systematic. Here we point out an almost complete correspondence between the
two, i.e. the cycle index of the symmetric group corresponds to the
probability mass function of ID on $\Z_+$ (Remark \ref{rem:relation} (ii)). 
This correspondence is justified by the uniqueness of the GF. All presented
results of Section \ref{main} stem essentially from this correspondence. We expect that
the relation gives a perspective to previous miscellaneous results and promote
further exchanges of the two fields. 

Note that we could give results only by combinatorial methods. However,
since the relation between combinatorics and probability and statistics
are interesting and worth describing, we present both
approaches. 


\section{Main results}\label{main}
We consider the probabilistic proof of 
\cite[Theorem 1]{bender:canfield:1996}, which clarifies the relation
between non-negative log-concave sequences and the cycle index polynomials. 
By the probabilistic proof, we extended the relation to that with non-negative
log-convex sequences. Furthermore, we give a combinatorial proof of the
extended part which yields additional new results about the cycle index polynomials.

Our focus is on the following extended theorem. 

\begin{theorem}
\label{thm:main}
 Let $(c_k)_{k\ge 1}$ be a sequence of non-negative real numbers such
 that $\sum_{k=1}^\infty (c_k/k)\, u_0^k$ exists for some $u_0>0$
 and define the sequence
 $(a_k)_{k\ge 0}$ and $(b_k)_{k\ge 0}$ by \eqref{eq:main-relation}. 
 Suppose that $(c_k)_{k\ge 1}$ is log-concave and $c_1^2\ge c_2$, then 
\begin{align}
 a_{k-1}a_{k+1} &\le a_k^2 \le \frac{k+1}{k} a_{k-1}a_{k+1}, \label{ineq:a_k}\\
 b_{k-1}b_{k+1} &\ge b_k^2 \ge \frac{k}{k+1}b_{k-1}b_{k+1}. \label{ineq:b_k}
\end{align}
 Moreover, the left ineq. in \eqref{ineq:a_k} or equivalently the right
 ineq. of \eqref{ineq:b_k} 
 holds only if $c_1^2\ge c_2$. \\
 Conversely, suppose that $(c_k)_{k\ge 1}$ is log-convex, then 
\begin{align}
 a_{k-1}a_{k+1} &\ge a_k^2, \label{ineq:a_k:convex} \\
 b_k^2 &\le \frac{k}{k+1}b_{k-1}b_{k+1} \label{ineq:b_k:convex}
\end{align}
 holds if and only if $c_1^2\le c_2$. 
\end{theorem}

\begin{remark}
 We mention the relation between Theorem \ref{thm:main} and
 \cite[Theorem 1]{bender:canfield:1996}. \\
$(i)$ The extended part is the convexity results of $(a_k)_{k\ge 0}$ : 
 \eqref{ineq:a_k:convex} or equivalently \eqref{ineq:b_k:convex},
 and ``only if'' parts in the log-concave case. \\
$(ii)$ A certain kind of summability condition for $(c_k)_{k\ge 1}$ is additionally
 assumed in Theorem \ref{thm:main}.\\
$(iii)$ Log-concavity assumption of
 $(1,c_1,c_2,\ldots)$ in \cite[Theorem 1]{bender:canfield:1996} and that
 of $(c_k)_{k\ge 1}$ together with $c_1^2\ge c_2$ are the same. Since
 the latter 
 seems more clear to mention ``only if'' condition, we adopt
 this. The same argument holds in log-convex case. \\
$(iv)$ We can not recover the structural result \cite[Theorem
 2]{bender:canfield:1996} for $(b_k)$ from probabilistic proof, i.e. 
 it seems difficult to 
 prove that 
 $(n+1)b_mb_n-mb_{m-1}b_{n+1}$ for $1\le m\le n$ could be
 expressed as a polynomial in
 $\mathscr{Y}=\{c_1,c_2,\ldots\}\,\cup\,\{c_jc_k-c_{j-1}c_{k+1}:0<j\le
 k\}$ with non-negative integer coefficients. 
\end{remark}

We begin to see the probabilistic proof of Theorem \ref{thm:main}.
For $\Z_+=\{0,1,2,\ldots\}$ let $(X_i)_{i=1,2,\ldots}$ be an independent and identically distributed
sequence of $\Z_+$
-valued random variables (r.v.'s for short). Let $N$ be Poisson r.v. with parameter $\lambda>0$ of which
probability is $P(N=k)=\frac{\lambda^k}{k!}e^{-\lambda},\,k\in\Z_+$.
We consider CP r.v. $S_N:= \sum_{j=1}^N X_j$. 
Writing $P_n=P(S_N=n),\,n\in\Z_+$ and $f_n:=P(X_1=n)$, we have an expression of 
probability GF of $S_N$ by  
\begin{align}
\label{eq:mgf:cp}
 \sum_{k=0}^\infty P_ku^k =\exp\Big(
 \lambda\big( \sum_{j=0}^\infty f_ju^j-1\big)
\Big) = e^{\lambda(f_0-1)} \exp\Big(
 \sum_{j=1}^n \frac{\lambda jf_j}{j}u^j\Big),\quad |u| \le u_0. 
\end{align}
Here putting $a_k=e^{-\lambda(f_0-1)}P_k$ and  $c_k=\lambda kf_k$, the
relation \eqref{eq:main-relation} is recovered. Then, once log-concavity
(resp. log-convexity) of $(P_k)$ is proven, that of $(a_k)$
follows.  

In the probabilistic proof, the relations \eqref{eq:mgf:cp} is
not directly used and our main tool is the following well-known
recursion for $(P_k)$\footnote{See e.g. \cite[Theorem
2.2]{sundt:vernic:2009} where the proof for the case $f_0=0$ is given, but it
is almost the same proof for the case $f_0\neq 0$. See also
\cite[Theorem 3.3.9]{mikosch:2009}. The recursion
\eqref{eq:panjer:rec} is studied independently in different fields, see 
\cite[pp.59,60]{sundt:vernic:2009} for
a bit long history.}: 
\begin{align}
\label{eq:panjer:rec}
\begin{split}
 P_0 =e^{\lambda(f_0-1)}, \qquad (n+1)P_{n+1} = \sum_{k=0}^n \lambda
 (k+1) f_{k+1}P_{n-k},  
\end{split}
\end{align}
which yields two key results used in the proof. 

One is the following theorem, an explanation of which is given in
Appendix for readers' sake. 
\begin{theorem}[Hansen {\cite[Theorem 1 and 2]{hansen:1988}}]
\label{thm:hansen}
 Let $(P_n)_{n\ge 0}$ and $(f_k)_{k\ge 0}$ be connected by
 \eqref{eq:panjer:rec}. Assume $(kf_k)_{k\ge1}$ is log-concave
 $($resp. log-convex$)$, then $(P_n)$ is log-concave $($resp. log-convex$)$ if
 and only if $\lambda f_1^2-2f_2\ge 0$ $($resp. $\lambda f_1^2-2f_2\le
 0$$)$. 
\end{theorem}
The other one is nearly log-convex result when $(P_n)$ is
log-concave. Although the result follows from a more general result
\cite[Lemma 2]{Schirmacher:1999},  
we give a short proof for convenience.  
\begin{lemma}
\label{lem:log-concave}
 Let $(P_k)_{k\ge 0}$ and $(f_k)_{k\ge 0}$ be related by
 \eqref{eq:panjer:rec}. If the assumption of log-concavity in Theorem
 \ref{thm:hansen} is satisfied, then $(P_k)$ further fulfills 
\begin{align*}
P_{n-1}P_{n+1} \le  P_n^2 \le \frac{n+1}{n}P_{n-1}P_{n+1}.
\end{align*}
\end{lemma} 

\begin{proof}
 Due to the relation \eqref{eq:panjer:rec}, we have 
 \begin{align*}
& (n+1)P_{n-1}P_{n+1}-nP_n^2 \\
&= \lambda(n+1) f_{n+1} P_0P_{n-1} +\sum_{k=0}^{n-1} \lambda(k+1)f_{k+1} 
(P_{n-k}P_{n-1}-P_{n-1-k}P_n). 
 \end{align*}
Since $(P_n)_{n\ge 0}$ is a log-concave sequence, the right-hand side is non-negative.
\end{proof}

\begin{proof}[Proof of Theorem \ref{thm:main}]
 By the assumption \eqref{eq:main-relation} is well-defined for
 $u\in[0,u_0]$. We notice that assumptions and results are unchanged if
 we multiply $(a_k,b_k,c_k)$ by $u_0^k$ and consider sequences
 $a_k'=a_ku_0^k,\,b_k'=b_ku_0^k,\,c_k'=c_ku_0^k$ such that they satisfy
 \eqref{eq:main-relation} with $u\in[0,1]$. Therefore, without loss of
 generality we may assume that $(c_k/k)_{k\ge 1}$ are infinitely
 summable. 
 We write the relation \eqref{eq:main-relation} in the form of \eqref{eq:mgf:cp}.
 Noticing $\sum_{k=1}^\infty (c_k/k)<\infty$, we take some 
 $\lambda > \sum_{k=1}^\infty c_k/k$ since $\lambda$ could
 be any positive constant. Let $c_k=\lambda f_k
 k,\,k\ge1$ in \eqref{eq:main-relation} and put 
 $f_0=1-\sum_{k=1}^\infty f_k$. Then multiply
 both sides by $e^{\lambda(f_0-1)}$ to obtain 
 \[
  \sum_{k=0}^\infty e^{\lambda(f_0-1)}a_k u^k
 =e^{\lambda(f_0-1)}\exp\Big(
 \sum_{j=1}^\infty \frac{\lambda j f_j}{j}u^j 
\Big). 
 \]
Since the right-hand side is probability GF of $P_k$ by
 the uniqueness $P_k=e^{\lambda(f_0-1)}a_k$ holds.  

 Now we check conditions of Theorem \ref{thm:hansen} in terms of $(c_k)$. The
 log-concavity (resp. log-convexity) of $(c_k)_{k\ge1}$ and that of
 $(kf_k)_{k\ge1}$ are equivalent. Moreover $\lambda f_1^2-2f_2 \ge
 1/\lambda \cdot (
 c_1^2-c_2\cdot 1) \ge 0$. Thus conditions follow from those of Theorem
 \ref{thm:main}. Then conclusions of Theorem
 \ref{thm:main} 
 are implied by 
 \[
 a_k^2-a_{k-1}a_{k+1} = e^{-2\lambda(f_0-1)}(P_k^2-P_{k-1}P_{k+1})
\]
 together with results of Theorem \ref{thm:hansen}. The second inequality of \eqref{ineq:a_k} is implied by Lemma
 \ref{lem:log-concave}. Finally \eqref{ineq:b_k} follows from
 \eqref{ineq:a_k}.
\end{proof}

We give a remark about further relations between CP distribution and the
cycle index. 

\begin{remark}
\label{rem:relation}
 (i) The recursion \eqref{eq:panjer:rec} is equivalent to that for the cycle
 index polynomials $A(\Sigma_n)$ of symmetric group $\Sigma_n$ with
 variables $(x_1,x_2,\ldots,x_n)$: 
 \begin{align}\label{recursion:sym}
  A(\Sigma_0)=1\quad \mathrm{and}\quad A(\Sigma_n)= \frac{1}{n} \sum_{\ell=1}^n x_{\ell}A(\Sigma_{n-\ell})
 \end{align}
 (see e.g. Harary and Palmer \cite[p.120]{harary:palmer:1973}). Namely,
 scale-adjusted $A(\Sigma_{n+1})$ has a correspondence with $P_{n}$. This is rationalized by
 comparing their GFs.
 The GF of 
 $A(\Sigma_n)$ is given by \eqref{eq:main-relation}, i.e. 
 \begin{align} \label{eg:gene:sym}
  \sum_{n=0}^\infty A(\Sigma_n) t^n = \exp\big( \sum_{j\ge1}
  \frac{x_j}{j} t^j\big) 
  \end{align}
  (see e.g. \cite[Theorem 4.7.2]{wilf:1994} for an elementary
 proof). Assume that $\sum_{j\ge 1} x_j/j<\infty$. 
 Let $\lambda=\frac{1}{1-f_0}\sum_{k=1}^\infty x_k/k,\,f_k=x_x/(\lambda
 k),\,k\ge 1$ and $f_0=1-\sum_{k=1}^\infty f_k$ as before
 and then \eqref{eg:gene:sym}
 coincides with \eqref{eq:mgf:cp}. Hence by the uniqueness of GFs we
 conclude that 
 for any $A(\Sigma_{n+1})$ there exists probability mass
 function $P_n$, and stated properties of $P_n$ (Theorem
 \ref{thm:hansen} and Lemma \ref{lem:log-concave})
 hold true for $A(\Sigma_n)$. \\
(ii) Since CP distribution coincides with ID distributions on $\Z_+$
 (see \cite[Theorem 3.2, III]{steutel:2004} ), $A(\Sigma_n)$ is relevant to ID distributions. Indeed, it is
 known that a distribution $P_k,\,k\in\Z_+$ with $P_0>0$ is ID if and
 only if the quantity $r_k$ with $k\in Z_+$ determined by 
 \begin{align}
\label{recur:id}
 P_{n+1} = \frac{1}{n+1} \sum_{k=0}^n r_k P_{n-k}
 \end{align} 
 are non-negative (see \cite[Theorem 4.4, II]{steutel:2004}  or
 \cite[Corollary 51.2]{sato:1999}). In view of \eqref{recur:id} and
 \eqref{recursion:sym}, one see the correspondence between the cycle index of
 symmetric group 
 and ID distributions. In each topic there are established
 properties. Thus further investigation of the relation would be our next
 interest.  
\end{remark}
Next we consider a combinatorial proof of the extended part based on the
cycle index polynomials as done in \cite{bender:canfield:1996}. 
Let
$\Sigma_m$ denote the symmetric group of degree $m$ and let $N_j(\sigma)$ be the
number of $j$-cycle in the permutation $\sigma$. Then the cycle index
polynomials $(a_m)$ and related polynomials $(b_m)$ are defined as  
\[
 a_m(c_1,c_2,\ldots,c_m)=\frac{1}{m!} b_m(c_1,c_2,\ldots,c_m) =
 \frac{1}{m!}\sum_{\sigma\in \Sigma_m} wt(\sigma), 
\]
where $wt(\sigma)=c_1^{N_1(\sigma)}\cdots c_m^{N_m(\sigma)}$. 
In what follows, we give several properties as derived in
\cite{bender:canfield:1996} which are used. 
For $\sigma_1\in \Sigma_{m+1}$ let
$\sigma_1'$ be $\sigma_1$ with $m+1$th element deleted from the cycle
containing it. The summation of $wt(\sigma_1')$ over all $\sigma_1\in
\Sigma_{m+1}$ yields 
\begin{align}
 \sum_{\sigma_1\in \Sigma_{m+1}} wt (\sigma_1') =(m+1) b_m. \label{eq:prop2}
\end{align}
If $m+1$ element belongs to a $j$-cycle of $\sigma_1$,
then 
\begin{align}
 c_{j-1}wt(\sigma_1)= c_j wt(\sigma_1') \label{eq:prop22}
\end{align} 
holds. 
We have two formulas for $b_{m+1}$ and $(m+1)b_m$: let $(m)_
k$ denote the falling factorial $m(m-1)\cdots (m-k+1)$, then 
\begin{align}
\label{eq:prop3}
 b_{m+1}&=\sum_{j=1}^{m+1}(m)_{j-1} c_j b_{m+1-j}, \\ 
 (m+1)b_{m} &= \sum_{j=1}^{m+1} (m)_{j-1} c_{j-1} b_{m+1-j},\label{eq:prop4}
\end{align}
both of which have combinatorial interpretations. 
For \eqref{eq:prop3}, $j$th term in the sum, $(m)_{j-1}c_j b_{m+1-j}$
implies the sum of $wt(\sigma_1)$ over all combinations in $\sigma_1$ such
that $j$-cycle contains $m+1$th element. There are $(m)_{j-1}$ ways to
construct $j$-cycle which contains $m+1$th element and $b_{m+1-j}$ is the
sum of weights over all permutations for remaining $m+1-j$ elements. Here
$c_j$ is the weight of $j$-cycle. For \eqref{eq:prop4}, $j$th term
implies the sum of $wt(\sigma_1')$ over all permutations in $\sigma_1$
such that $m+1$th element is removed from $j$-cycle of $\sigma_1$, so
that $c_j$ of $(m)_{j-1}c_j b_{m+1-j}$ in \eqref{eq:prop3} is replaced by
$c_{j-1}$. Here we additionally use \eqref{eq:prop2}. 

Now we give a combinatorial proof of the extended part. We get the idea
from the formula (5) by \cite{hansen:1988}, though an analogue of the
formula has already been used in Bruijn and Erd\"os
\cite{bruijn:erdos:1953}. 

\begin{proof}
 Our goal is to prove \eqref{ineq:b_k:convex} by the induction, and then \eqref{ineq:a_k:convex} follows by the equivalence. 
 Since $b_0=1,\,b_1=c_1$ and $b_2=c_1^2+c_2$, we
 have $b_0b_2-2b_1^2=c_2-c_1^2\ge 0$. Assume that
 \eqref{ineq:b_k:convex} holds with $k=m-1$ and then we consider  
\begin{align*}
& c_m(mb_{m-1}b_{m+1}-(m+1)b_m^2) \\
&= mb_{m-1}(c_mb_{m+1}-c_{m+1}(m+1)b_m)-(m+1)b_m(c_mb_m-c_{m+1}mb_{m-1}).
\end{align*}
From the properties \eqref{eq:prop3} and  \eqref{eq:prop4} 
that 
\begin{align*}
& mb_{m-1}\big(
 \sum_{j=1}^{m+1} c_{m} (m)_{j-1} c_j b_{m+1-j}- c_{m+1}(m)_{j-1}c_{j-1}b_{m+1-j}
\big) \\
&\quad -(m+1)b_m \big(
\sum_{j=1}^m c_m(m-1)_{j-1}c_j b_{m-j}-c_{m+1}(m-1)_{j-1}c_{j-1}b_{m-j}
\big) \\
&= mb_{m-1}
\sum_{j=1}^m (m)_{j-1}b_{m+1-j}(c_mc_j-c_{m+1}c_{j-1})  \\
&\quad - mb_m 
\sum_{j=1}^m (m-1)_{j-1}b_{m-j} (c_mc_j-c_{m+1}c_{j-1})
+ b_m \sum_{j=1}^m (m-1)_{j-1}b_{m-j}(c_{m+1}c_{j-1}-c_mc_j) \\
& =m\sum_{j=1}^m (c_{m+1}c_{j-1}-c_mc_j) (m-1)_{j-2} \big\{
(m+1-j)b_m b_{m-j}-mb_{m-1}b_{m+1-j}
\big\} \\
&\quad + b_m \sum_{j=1}^m (m-1)_{j-1} b_{m-j}(c_{m+1}c_{j-1}-c_mc_j).
\end{align*} 
Now since 
\eqref{ineq:b_k:convex} 
holds for $k=1,\ldots,m-1$ by
 the induction hypothesis, 
  one could see from the last expression that \eqref{ineq:b_k:convex} is satisfied
 also with $m$. 
\end{proof}

By combinatorial approach, we can prove further results more than 
\eqref{ineq:a_k}--\eqref{ineq:b_k:convex}. The following is an
extension of Theorem 2 in \cite{bender:canfield:1996} which is related
with the log-convexity result in Theorem \ref{thm:main}. 

\begin{theorem}
 \label{thm-main2}
 Let $c_0=1$ and let $c_1,c_2,\ldots$ be indeterminates. Further let 
 \begin{align*}
 & \mathscr{X}= \{c_1,c_2,\ldots\}, \\
 & \mathscr{Y}=\mathscr{X}\cup \{c_jc_k-c_{j-1}c_{k+1} : 0<j \le k\}, \\
 &  \mathscr{Z}= \mathscr{X}\cup \{c_{j-1}c_{k+1}-c_jc_k : 0<j\le k\}
 \end{align*}
and define the sequence $(b_k)_{k\ge 0}$ by \eqref{eq:main-relation}. 
 Then 
 \begin{align}
   & (n+1)b_mb_n -m b_{m-1}b_{n+1} \in \N[\mathscr{Y} ], \label{poly1} \\
   & m b_{m-1}b_{n+1} -(n+1) b_mb_n \in \N[\mathscr{Z}]/\N[\mathscr{X}] \label{poly2}
  \quad \mathrm{for}\quad 1\le m \le n,
 \end{align}
 namely, $(n+1)b_mb_n-mb_{m-1}b_{n+1}$ can be expressed as a polynomial in
 $\mathscr Y$ with non-negative integer coefficients and
 $mb_{m-1}b_{n+1}-(n+1)b_mb_n$ can be expressed as a ratio of 
 polynomials in $\mathscr{Z}$ and those in $\mathscr{X}$ with
 respectively
 non-negative integer coefficients.  
\end{theorem}

 Note that proof for \eqref{poly1} is given in
 \cite{bender:canfield:1996} and further extended by
 \cite{Schirmacher:1999}, where $b_mb_n-b_{m-1}b_{n+1}\in
 \N[\mathscr{Y}]$ is shown. Therefore we give a proof only for
 \eqref{poly2}.

\begin{proof}
 We show \eqref{poly2} by the induction. It is immediate to see 
 \[
  b_0b_2-2b_1^2 = c_2-c_1^2 \in \N[\mathscr{X}].
 \]
 Assume that 
 \begin{align}
 \label{eq:induction}
  kb_{k-1} b_{k+1} - (k+1)b_k^2 \in \N[\mathscr{Z}]/\N[\mathscr{X}]
 \end{align}
 holds for $k\le m-1$, then 
 \begin{align}\label{eq:induction0}
  kb_{k-1} b_{\ell+1}-(\ell+1)b_kb_{\ell}\in
  \N[\mathscr{Z}] /\N[\mathscr{X}] 
 \end{align}
 holds for all $1\le k\le \ell \le m-1$. Indeed, we observe that 
 \begin{align}
\label{eq:induction1}
  \frac{b_{\ell+1}}{(\ell+1)b_{\ell}} - \frac{b_k}{kb_{k-1}} = \Big(
 \frac{b_{\ell+1}}{(\ell+1)b_{\ell}} - \frac{b_{\ell}}{\ell b_{\ell-1}} 
 \Big) 
 +\cdots + \Big(
 \frac{b_{k+1}}{(k+1)b_k}- \frac{b_k}{kb_{k-1}}
 \Big)
 \end{align}
is included in $\N[\mathscr{Z}]/\N[\mathscr{X}]$ by the induction
 hypothesis. Here we also use $b_k \in \N[\mathscr{X}],\,k\le m-1$ which
 follows from the recursion \eqref{eq:prop3}.
 Then we multiply
 \eqref{eq:induction1} by $(\ell+1)b_\ell\cdot kb_{k-1}$ to conclude
 \eqref{eq:induction0}.  Now we recall the previous equality in the proof
 of Theorem \eqref{thm:main} : 
 \begin{align*}
  & c_m ( mb_{m-1}b_{m+1}-(m+1)b_m^2) \\
  & = m\sum_{j=1}^m (m-1)_{j-2}(c_{m+1}c_{j-1}-c_mc_j)\big\{
 (m+1-j)b_{m-j}b_m -m b_{m-1}b_{m+1-j}
 \big\} \\
  & + b_m \sum_{j=1}^m(m-1)_{j-1} b_{m-j}(c_{m+1}c_{j-1}-c_mc_j). 
 \end{align*}
Then since \eqref{eq:induction0} holds for $1\le k \le \ell\le m-1$ by
 the assumption we conclude \eqref{eq:induction} with $k=m$. 
\end{proof}

\begin{remark}
 We do not know whether $mb_{m-1}b_{n+1}-(n+1)b_mb_n \in
 \N[\mathscr{Z}],\,1\le m\le n$ holds or not. Our conjecture is negative but we do not
 prove it. 
\end{remark}

\section{Convolution and binomial convolution of cycle index
 polynomials}\label{sec:convo}


We first review a generalization of the Bender and Canfield exponential
formula (Theorem 2 in \cite{bender:canfield:1996}) to convoluted
sequences which is done by Schirmacher \cite{Schirmacher:1999}. At there
we give resulting consequences of the
extension together with its alternative proof. Then we investigate the
``log-convex counter part''. These generalizations to normal and binomial
convolutions may provide interesting examples of log-concave/convex
sequences which are not included in known laws. 

\begin{theorem}[Theorem 3 in \cite{Schirmacher:1999}]
\label{thm:conv:logconcave}
 For $i=1,\ldots,w,\,w\in \N$ let $c_{i,0}=1$ and let
 $c_{i,1},c_{i,2},\ldots$ be indeterminates. Let 
\[
 \mathscr{Y}_i = \{c_{i,1},c_{i,2},\ldots\} \cup
 \{c_{i,j}c_{i,k}-c_{i,j-1}c_{i,k+1}: 0<j \le k \}
\]
and let 
\begin{align}
\label{eq:convexp}
 \sum_{k=0}^\infty A_k u^k = \sum_{k=0}^\infty \frac{B_k u^k}{k!}
 =\exp\Big(
 \sum_{j=0}^\infty \frac{\sum_{i=1}^w c_{i,j}}{j}u^j
\Big).
\end{align}
Then for $1 \le m \le n$,  
\begin{align}
 \label{eq:conv-ccv1}
 A_mA_n-A_{m-1}A_{n+1} & \in \R_+ [\cup_{i=1}^w \mathscr{Y}_i], \\ 
 \label{eq:conv-ccv2}
 (n+1) A_{m-1}A_{n+1} -mA_mA_n & \in  \R_+ [\cup_{i=1}^w \mathscr{Y}_i].
\end{align}
 Namely $A_mA_n-A_{m-1}A_{n+1}$ are polynomials in $\cup_{i=1}^w
 \mathscr{Y}_i$ with non-negative rational coefficients and so are
 $(n+1) A_{m-1}A_{n+1} -mA_mA_n$. 
 The relations \eqref{eq:conv-ccv1} and \eqref{eq:conv-ccv2} are
 equivalent to 
 \begin{align}
 \label{eq:conv-ccv3}
 (n+1)B_mB_n-mB_{m-1}B_{n+1} & \in \R_+ [\cup_{i=1}^w \mathscr{Y}_i],  \\ 
 \label{eq:conv-ccv4}
  B_{m-1}B_{n+1} -B_mB_n & \in  \R_+ [\cup_{i=1}^w \mathscr{Y}_i].
\end{align}
\end{theorem}

The formula \eqref{eq:convexp} implies that $(A_k)_{k\ge0}$ and
$(B_k)_{k\ge 0}$ are normal and binomial convolutions
respectively (\cite[Chapters 2.2 and 2.3]{wilf:1994}).

In \cite{Schirmacher:1999} a general proof based on Cauchy-Binet
formula was just suggested (see also \cite{stanley:1989}) and
explicit expressions were skipped. In order to see clearly the reason why
log-concavity is preserved and log-convexity is not in the normal convolution
we give a direct proof. 
The following lemma is
crucial of which proof is given in Appendix.  

\begin{lemma}
\label{lem:dmdn}
 Let $(x_k)_{k\ge0}$ and $(y_k)_{k\ge0}$ are indeterminates. Let 
 $D_n=\sum_{k=0}^n x_k y_{n-k}$, then for $1\le m\le n$, 
  \begin{align}
\label{eq:dmdn}
 \begin{split}
 & D_mD_n-D_{m-1}D_{n+1} \\
 & = x_0y_0x_my_n+ y_0 \sum_{k=0}^{n-m-1} x_m x_{k+1} y_{n-1-k} 
 +y_0
  \sum_{k=0}^{m-1} y_k (x_mx_{n-k}-x_{n+1}x_{m-k-1}) \\
 &\ + \sum_{k=0}^{m-1} \sum_{\ell=0}^{n-m} x_k
  x_{\ell}(y_{m-k}y_{n-\ell}-y_{m-k-1}y_{n+1-\ell}) \\
 &\ + \sum_{k=0}^{m-1} \sum_{\ell=0}^k (x_kx_{n-m+1+\ell}-x_\ell
  x_{n-m+1+k})(y_{m-k}y_{m-1-\ell}-y_{m-1-k}y_{m-\ell}). 
 \end{split}
 \end{align}
 Moreover, putting 
 \begin{align*}
  \mathscr{G} &=\{x_0,x_1,\ldots,x_n\}\cup \{x_mx_n-x_{m-1}x_{n+1}:1\le
  m\le n\}, \\
  \mathscr{F} &= \{y_0,y_1,\ldots,y_n\}\cup \{y_my_n-y_{m-1}y_{n+1}:1\le
  m\le n\},
 \end{align*}
 we obtain 
 \[
  D_mD_n-D_{m-1}D_{n+1} \in \N[\mathscr{G}\cup \mathscr{F}],\quad 1\le
 m\le n.
 \]
\end{lemma} 

\begin{proof}[Proof of Theorem \ref{thm:conv:logconcave}]
Observe that $A_k$ is $w$th convolution of sequence
 $(a_{i,k}),\,i=1,2,\ldots,w$ such that 
 \[
  \sum_{k=0}^\infty a_{i,k}u^k = \exp\Bigg( \sum_{j=0}^\infty
 \frac{c_{i,j}}{j} u^j
\Bigg)
 \]
 holds. By the induction, it suffices to prove the result
 for the bivariate convolution. Let $x_i\in\R_+ [\mathscr{Y}_1]$ and
 $y_j\in \R_+ [\mathscr{Y}_2]$ and suppose that for $1\le m\le n$
 \begin{align*}
  x_mx_n-x_{m-1}x_{n+1} \in \R_+ [\mathscr{Y}_1]\quad \mathrm{and}\quad
  y_my_n-y_{m-1}y_{n+1} \in \R_+ [\mathscr{Y}_2].
 \end{align*}
 Then due to Lemma \ref{lem:dmdn}, the bivariate convolution $D_n=\sum_{k=0}^n x_k
 y_{n-k},\,n\ge 0$ satisfies  
 \[
  D_mD_n - D_{m-1}D_{n+1} \in \R_+ [\mathscr{Y}_1 \cup \mathscr{Y}_2],\quad
 1 \le m\le n. 
 \]
Therefore \eqref{eq:conv-ccv1} holds. Equation \eqref{eq:conv-ccv2}
 follows from Lemma $2$ of \cite{Schirmacher:1999}. 
\end{proof}

\begin{remark}[Remark for Lemma \ref{lem:dmdn}]
 $(i)$ Equation \eqref{eq:dmdn} would be calculated from the
 Cauchy-Binet formula. In the proof of \cite{Schirmacher:1999}
 only the formula is mentioned and the detail is omitted.  
 We confirm the derivation by applying the formula to 
 \begin{align*}
  A &= \left(
    \begin{array}{ccccc}
      y_0 & y_1 & \cdots  & y_n &y_{n+1} \\
      0 & y_0 &  \cdots & y_{n-1} &y_n 
    \end{array}
  \right), \\
 B &= \left(
    \begin{array}{cccccccc}
      x_m & x_{m-1} & \cdots & x_0 & 0 & \cdots  & 0 & 0 \\
      x_{n+1} & x_n & \cdots & x_{n-m+1} & x_{n-m} & \cdots &x_1 &x_0  
    \end{array}
  \right)'.  
 \end{align*}
 See also \cite{stanley:1989}. \\ 
 $(ii)$ Menon \cite{menon:1969} proved that log-concavity is preserved
 under convolution by a direct calculation, i.e. proved the case $n=m$ in
 Lemma \ref{lem:dmdn}. Our results with $n=m$ coincides
 the corresponding calculation in \cite{menon:1969}. \\
 $(iii)$ From \eqref{eq:dmdn} one could observe that if one of sequences
 $(x_k)$ and $(y_k)$ is log-convex, then both positive and negative terms
 appear in \eqref{eq:dmdn}. Accordingly we cloud not judge the sign of
 $D_mD_n-D_{m-1}D_{n+1}$ without assuming additional conditions. 
\end{remark}

\begin{remark}[Remark for Theorem \ref{thm:conv:logconcave}]
 Assume that $(c_{i,j})_{j\ge 1},\,i=1,\ldots,w$ are log-concave in $j$ with
 $c_{i,0}=1$. \\
 $(i)$ It is immediate from equations \eqref{eq:conv-ccv1} and
 \eqref{eq:conv-ccv3} that 
\[
  (n+1)B_nB_m- mB_{m-1}B_{n+1} \ge 0\ \Leftrightarrow\ A_mA_n-A_{m-1}A_{n+1}\ge 0. 
 \]
 Notice that the result
 of Theorem \ref{thm:conv:logconcave} is not covered by the original
 version (Theorem \ref{thm:main}) since the sum $C_j=\sum_{i=1}^w
 c_{i,j}$ in the relation 
 \[
   \sum_{k=0}^\infty A_k u^k = \sum_{k=0}^\infty \frac{B_k u^k}{k!}
 =\exp\Big(
 \sum_{j=0}^\infty \frac{C_j}{j} u^j
\Big)
 \]
 does not always log-concave (see example \ref{ex1}). Namely, operations
 of (normal or binomial) convolution to sequences which satisfy
 \eqref{eq:main-relation} widens the class of sequence $(c_j)$ from
 which log-concavity of $(a_k)$ follows. \\
 $(ii)$ The log-concavity of $(A_k)$ \eqref{eq:conv-ccv1} and the log-convexity of
 $(B_k)$ \eqref{eq:conv-ccv4} are concluded from general theory : 
 The log-concavity is preserved by both ordinary and binomial convolution
 while log-convexity is preserved only by binomial convolution (see
 p.455 in \cite{lin:wan:2007}). \\
 $(iii)$ We do not know whether $(n+1)B_mB_n-mB_{m-1}B_{n+1}\in
 \N[\cup_{i=1}^w \mathscr{Y}_i]$ follows as in \eqref{poly1} or not. Our conjecture is
 positive, but it remains to be seen. 
\end{remark}

 The following is log-convex counterpart. 

 \begin{corollary}
  \label{cor:cov-lconvex}
  For $w\in\N$ let 
 \[
  \sum_{k=0}^\infty A_k u^k = \sum_{k=0}^\infty \frac{B_k u^k}{k!}=
  \exp\Bigg(
 \sum_{j=1}^\infty \frac{\sum_{i=1}^w c_{i,j}}{j}u^j
 \Bigg).
 \]
 Suppose that $(c_{i,k})_{k\ge 1},\,i=1,\ldots,w$ are respectively
  log-convex in $k$. Then for $1\le m\le n$ 
 \begin{align}
 \label{eq:conv-convex}
  mB_{m-1}B_{n+1}-(n+1)B_mB_n\ge 0
 \end{align}
 if and only if 
 \begin{align}
 \label{eq:condi:conv-convex}
 \sum_{i=1}^w c_{i,2}- \Big(
  \sum_{i=1}^w c_{i,1}
 \Big)^2 \ge 0. 
 \end{align}
 Equation \eqref{eq:conv-convex} is equivalent to 
 \begin{align*}
  A_{m-1}A_{n+1} -A_mA_n \ge 0. 
 \end{align*}
 \end{corollary}

 \begin{proof}
  Let $C_0=1$ and $C_j=\sum_{i=1}^w c_{i,j}$. Put
  $\mathscr{V}=\{C_1,C_2,\ldots\}$ and 
 \[
  \mathscr{W}= \mathscr{V} \cup \{C_{j-1}C_{k+1}-C_j C_k : 0<j\le
  k\}.
 \]
 Then the relation \eqref{eq:main-relation} concludes via Theorem
  \ref{thm-main2} that 
 \[
  mB_{m-1}B_{n+1}-(n+1)B_mB_n\in
  \frac{\N[\mathscr{W}]}{\N[\mathscr{V}]}\quad \mathrm{for}\quad 1\le
  m\le n. 
 \] 
 Since log-convexity is preserved under the summation, $(C_k)_{k\ge 1}$
  is again log-convex. Moreover, since 
 \[
  C_0C_2- C_1^2 = \sum_{i=1}^w c_{i,2}-\Big(
 \sum_{i=1}^w c_{i,1}
 \Big)^2 \ge 0
 \]
 implies $C_0C_{k+1}-C_1C_k\ge 0,\ k\ge 1$, we conclude 
\[
 mB_{m-1}B_{n+1}-(n+1)B_mB_n\ge 0 \quad \Leftrightarrow \quad
  A_{m-1}A_{n+1}-A_m A_n \ge 0 
\]
 follows. Notice that $B_0=1,\,B_1=C_1$ and $B_2=C_2+C_1^2$ imply 
 \[
  B_0B_2-B_1^2= C_0C_2-C_1^2, 
 \]
 which proves the 'only if' part. 
 \end{proof}
 
 \begin{remark}
  $(i)$ Log-convexity is generally not preserved by
  convolution. However, under the condition \eqref{eq:condi:conv-convex} of Corollary
  \ref{cor:cov-lconvex} it is
  preserved. \\
  $(ii)$ Log-convexity is preserved under binomial convolution. The
  result \eqref{eq:conv-convex} in Corollary
  \ref{cor:cov-lconvex} is slightly stronger but with the condition
  \eqref{eq:condi:conv-convex}. 
 \end{remark}

 In what follows, we see examples of convoluted sequences which
 reflect obtained theories in this section. We use statistical
 distributions for construction of sequences. 
 
 \begin{example}[Geometric distribution sequence]
 \label{ex1}
  Let $c_k=p(1-p)^{k-1},\,k=1,2,\ldots,p\in(0,1)$ and $c_0=1$ in
  \eqref{eq:main-relation}. We observe that $c_k^2-c_{k-1}c_{k+1}=0$, so that $(c_k)$ is both
  log-concave and log-convex. 
  Due to Theorem \ref{thm:main} together with $c_1^2-c_0c_2=p(2p-1)$, 
  the sequence $(a_k)$
  is log-concave $($resp. log-convex$)$ for $p\ge 1/2$ (resp. $p\le
  1/2$). Now in \eqref{eq:convexp} put $c_{1,k}=p(1-p)^{k-1}$ with $c_{1,0}=1$ and 
  prepare another Geometric sequence by
  $c_{2,k}=p'(1-p')^{k-1},\,p'\in(0,1),\,c_{2,0}=1$ and further define
  $C_k=c_{1,k}+c_{2,k}$. Note that $(C_k)_{k\ge1}$ is log-convex since 
  \[
   C_k^2-C_{k-1}C_{k+1}=-pp'(1-p)^{k-1}(1-p')^{k-1}(2-p-p')^2 \le 0.
  \]
  However, from Theorem \ref{thm:conv:logconcave} if $p,p' \ge 1/2$,
  $(A_k)$ of \eqref{eq:convexp} is log-concave, the result is not included in the original
  Theorem \ref{thm:main}. On the other hand even if $p,p'\le 1/2$,
  $(A_k)$ is not always log-convex since 
  \[
   C_1^2-C_0C_2 = (p+p')(2(p+p')-1)-2pp' >0
  \]
  for $p=p'=1/2$. Furthermore if $p+p'\le 1/2$, then $(A_k)$ is
  log-convex, so that this provides an example such that log-convexity
  is preserved under normal convolution. 
 \end{example}

 \begin{example}[Log-series distribution sequence]
  Prepare $n$ sequences $c_{i,k}=p_i^k/(k\log
  (1-p_i)),\,p_i\in(0,1),\,i=1,2,\ldots,n,\,k=1,2,\ldots$, each of which
  is known to be a log-series distribution. For each $i$, 
 \[
  c_{i,k}^2-c_{i,k-1}c_{i,k+1} \le 0, 
 \]
 so that $(c_{i,k})$ are log-convex sequences in $k$. We put $c_k=c_{1,k}$ in the
  relation \eqref{eq:main-relation}. Then Theorem \ref{thm:main}
  concludes log-convexity of $(a_k)$ if $p_1 \ge 1-e^{-2}$ since 
 \[
  c_1^2 -c_0c_2 = \frac{p_1^2}{2\log^2 (1-p_1)}(2+\log (1-p_1)). 
 \]
 Now we consider the $n$th convolution and put $C_k=\sum_{i=1}^n
  c_{i,k}$ in \eqref{eq:convexp}. Here $(C_k)_{k\ge0}$ with $C_0=1$ is log-convex and $(A_k)$
  is $n$th convolution of the corresponding $(a_{i,k})'s,\,i=1,2,\ldots,n.$ Since 
 \begin{align*}
  C_1^2 - C_0C_2 &= (\sum_{i=1}^n c_{i,1})^2 - \sum_{i=1}^n c_{i,2} \\
 & \le n \sum_{i=1}^n c_{i,1}^2 -\sum_{i=1}^n c_{i,2} \\
 & \le \sum_{i=1}^n \frac{p_i^2}{2\log^2(1-p_i)}(2n+\log (1-p_i)) \le 0 
 \end{align*} 
 if $p_{i,k}\ge 1-e^{-2n}$ respectively, $(A_k)$ is shown to be
  log-convex. This is an example such that convolution of log-convex
  sequences yields again log-convex sequences. 
 
 \end{example}



\appendix 
\section{Proof of Lemma \cite{Schirmacher:1999}}
By\ adjusting number of components in sums and changing subscripts, we
have 
\begin{align*}
& D_mD_n-D_{m-1}D_{n+1} \\
& = x_0(y_n D_m-y_{n+1}D_{m-1})\ (:=\mathrm{I}) \\
& \quad +y_0\Big(
 x_m\sum_{k=0}^{n-1}x_{k+1}y_{n-1-k}-x_{n+1} \sum_{k=0}^{m-1}x_ky_{m-1-k}
\Big)\ (:=\mathrm{II}) \\
& \quad + \sum_{k=0}^{m-1} \sum_{\ell=0}^{n-1}
x_{k} x_{\ell+1} (y_{m-k}y_{n-1-\ell}-y_{m-1-k}y_{n-\ell})\ (:=\mathrm{III}).
\end{align*}
Equation $(\mathrm{I})$ corresponds to the first term and the sum with
$\ell=0$ in the $4$th term of 
\eqref{eq:dmdn}.  Moreover, we observe that
\begin{align*}
 \mathrm{II} &= y_0 x_m \sum_{k=0}^{n-m-1}x_{k+1}y_{n-1-k} + y_0
 \sum_{k=0}^{m-1} y_{m-1-k}(x_mx_{n-m+1+k}-x_{n+1}x_k).
\end{align*}
Then by reversing indices of k in the latter sum, the expression is shown to coincide with 
sum of the second and the third terms of \eqref{eq:dmdn} respectively. Again by
shifting indices, we obtain 
\begin{align*}
 \mathrm{III} &= \sum_{k=0}^{m-1} \sum_{\ell=1}^{n-m}
 x_kx_{\ell} (y_{m-k}y_{n-\ell}-y_{m-1-k}y_{n-\ell+1}) \\
 & \quad + \sum_{k=0}^{m-1}
 \sum_{\ell=0}^{m-1}x_kx_{n-m+1+\ell}(y_{m-k}y_{m-1-\ell}-y_{m-1-k}y_{m-\ell}). 
\end{align*}
The first sum is equivalent to the sum with $\ell\ge 1$ in $4$th sum of \eqref{eq:dmdn}, and due to the symmetry of
indices, the second sum is written as the $5$th sum of 
\eqref{eq:dmdn}. Now the proof is over. \hfill $\Box$

\section{Results in Hansen \cite{hansen:1988}}

Our results are heavily depends on Hansen \cite[Lemmas 1,2 and Theorems
1,2]{hansen:1988}. In \cite{hansen:1988}, proofs are sometimes
omitted.
For readers' sake we restate results and review proofs recovering
omitted parts. For convenience write $r_k=\lambda(k+1)f_{k+1}$ so that
the recursion \eqref{eq:panjer:rec} has 
\begin{align}
\label{eq:omit:reccursion}
 (n+1) P_{n+1} = \sum_{k=0}^n r_k P_{n-k}. 
\end{align}
We are starting with two Lemmas. 

\begin{lemma}[Lemma 2 of \cite{hansen:1988}]
\label{app:lem1}
 Assume \eqref{eq:omit:reccursion} and let $P_{-1}=0$. Then
\begin{align}
\label{eq:relaton:log-concave}
\begin{split}
 m(&m+2)(P_{m+1}^2 - P_mP_{m+2})\\ 
        &= P_{m+1}(r_0P_m-P_{m+1}) + \sum_{\ell=0}^m \sum_{k=0}^\ell
 (P_{m-\ell}P_{m-k-1}-P_{m-k}P_{m-\ell-1})(r_{k+1}r_\ell-r_{\ell+1}r_k),  
\end{split}
\end{align}
\begin{align}
\label{eq:relaton:log-convex}
\begin{split}
 r_{m+1}&(m+2)(P_{m+1}P_{m+3}-P^2_{m+2}) \\
&= P_{m+1}(r_{m+2}P_{m+2}-r_{m+1}P_{m+3}) + \sum_{k=0}^m  
(P_{m-k}P_{m+2}-P_{m+1}P_{m-k+1})(r_{m+2}r_k-r_{k+1}r_{m+1}). 
\end{split}
\end{align}
\end{lemma}

\begin{proof}
For \eqref{eq:relaton:log-concave}, we have 
\begin{align*}
 & m(m+2)P^2_{m+1} -m P_m (m+2)P_{m+2} \\
 & = (m+1)^2 P_{m+1}^2 - mP_m(m+2)P_{m+2}-P_{m+1}^2 \\
 & = \Big(\sum_{k=0}^m r_k P_{m-k}\Big)^2 - \Big(
 \sum_{k=0}^{m-1} r_k P_{m-1-k}
\Big)\Big(
 \sum_{k=0}^{m+1}r_k P_{m+1-k} 
\Big)- P_{m+1}^2 \\
 & = \sum_{k=0}^{m} r_k P_{m-k} \sum_{\ell=1}^m r_\ell P_{m-\ell}
 -\sum_{k=0}^{m-1} r_k P_{m-1-k} \sum_{\ell=1}^{m+1} r_\ell P_{m+1-\ell}
 + r_{0}P_m P_{m+1}-P_{m+1}^2 \\
 &= P_{m+1}(r_0 P_m-P_{m+1}) + \sum_{k=0}^m \sum_{\ell=0}^m
 (r_kr_{\ell+1}-r_{k+1}r_\ell) P_{m-k}P_{m-1-\ell} =  \eqref{eq:relaton:log-concave},
\end{align*}
where we notice that terms of $k=\ell$ are zero in the sum. 
Next we see \eqref{eq:relaton:log-convex} and write 
\begin{align*}
 & r_{m+1}(m+2)(P_{m+1}P_{m+3}-P_{m+2}^2) \\
 & = r_{m+1}(m+3)P_{m+1}P_{m+3}-r_{m+1}(m+2)P_{m+2}^2 -
 r_{m+1}P_{m+1}P_{m+3} \\
 &\quad +(m+2)
 r_{m+2}P_{m+1}P_{m+2}-(m+2)r_{m+2}P_{m+1}P_{m+2} \\
 & = r_{m+1}(m+3)P_{m+1}P_{m+3}-r_{m+1}(m+2)P_{m+2}^2 +
 (m+1)r_{m+2}P_{m+1}P_{m+2} -(m+2)r_{m+2}P_{m+1}P_{m+2} \\
 &\quad +r_{m+2}P_{m+1}P_{m+2}-r_{m+1}P_{m+1}P_{m+3} \\
 & = r_{m+1}P_{m+1}\sum_{k=0}^{m+2}r_k P_{m+2-k} -
 r_{m+1}P_{m+2}\sum_{k=0}^{m+1}r_k P_{m+1-k} +
 r_{m+2}P_{m+2}\sum_{k=0}^m r_k P_{m-k}  -
 r_{m+2}P_{m+1}\sum_{k=0}^{m+1}r_k P_{m+1-k}\\
 & \quad
 -r_{m+1}P_{m+1}P_{m+3}+r_{m+2}P_{m+1}P_{m+2} \\
 & = r_{m+1} P_{m+1} r_0 P_{m+2} +r_{m+1}P_{m+1}r_{m+2}P_0 +
 r_{m+1}P_{m+1} \sum_{k=1}^{m+1}r_k P_{m+2-k} \\
 & \quad - r_{m+1}P_{m+1}r_0P_{m+2} -
 r_{m+1}P_{m+2}\sum_{k=1}^{m+1}r_kP_{m+1-k} \\
 & \quad -r_{m+2}P_{m+1}r_{m+1}P_0 -r_{m+2}P_{m+1}\sum_{k=0}^m r_k
 P_{m+1-k} +r_{m+2}P_{m+2}\sum_{k=0}^m r_k P_{m-k} \\
 &\quad + P_{m+1}(r_{m+2}P_{m+2}-r_{m+1}P_{m+3}) \\
 &= P_{m+1}(r_{m+2}P_{m+2}-r_{m+1}P_{m+3}) \\
 &\quad + r_{m+1}P_{m+1}\sum_{k=0}^m
 r_{k+1}P_{m+1-k}-r_{m+2}P_{m+1}\sum_{k=0}^m r_k P_{m+1-k} \\
 &\quad -r_{m+1}P_{m+2}\sum_{k=0}^m r_{k+1}P_{m-k}
 +r_{m+2}P_{m+2}\sum_{k=0}^m r_k P_{m-k} \\
 &= P_{m+1}(r_{m+2}P_{m+2}-r_{m+1}P_{m+3}) \\
 &\quad + \sum_{k=0}^m P_{m+1}P_{m+1-k}(r_{m+1}r_{k+1}-r_{m+2}r_k) +
 \sum_{k=0}^m P_{m+2}P_{m-k}(r_{m+2}r_k-r_{m+1}r_{k+1}) \\
 &= \eqref{eq:relaton:log-convex}. 
\end{align*}
\end{proof}

\begin{lemma}[Lemma 2 of \cite{hansen:1988}]
 Assume \eqref{eq:omit:reccursion} and $P_0>0$, then \\
$(i)$ if $(P_n)$ is strictly log-concave for $n=1,2,\ldots,m$, then
 $r_0P_m-P_{m+1}>0$, \\
$(ii)$ if $(r_n)$ is strictly log-convex and $r_0^2-r_1<0$, then
 $r_{m+2}P_{m+2}-r_{m+1}P_{m+3}>0$. 
\end{lemma}

\begin{proof}
 $(i)$ Since $(P_{n+1}/P_n)$ is decreasing, we have
 $r_0=P_1/P_0>P_{m+1}/P_m$. \\
 $(ii)$ Since $(r_{m+1}/r_m)$ is increasing, the recursion
 \eqref{eq:omit:reccursion} yields 
 \begin{align*}
  (m+3)P_{m+3} &= r_0P_{m+2} + \sum_{k=0}^{m+1} r_k P_{m+1-k}
  \frac{r_{k+1}}{r_k} \\
  &\le P_{m+2} \frac{r_{m+2}}{r_{m+1}} + (m+2) P_{m+2} \max_{1\le k\le
  m+2}\left\{
 \frac{r_k}{r_{k-1}}
\right\} \\
  &= (m+3)P_{m+2}\frac{r_{m+2}}{r_{m+1}}. 
 \end{align*}
\end{proof}

\begin{proof}[Proof of Theorem \ref{thm:hansen}]
 Log-concave part :
 Suppose $(r_k)$ (or equivalently $(kf_k)$) is strictly log-concave and
 $r_0^2-r_1\ge0$ ($\Leftrightarrow\, \lambda f_1^2 -2f_2$). Then 
 \begin{align*}
  2(P_1^2-P_0P_2)= P_0^2 (r_0^2-r_1)>0. 
 \end{align*} 
Now with this and Lemma \ref{app:lem1} $(i)$, we apply induction to
 \eqref{eq:relaton:log-concave} to see that $(P_n)$ is strictly
 log-concave. Since any log-concave sequence can be written as a limit
 of strictly log-concave sequences, the proof is completed. \\
 Log-convex part : The proof is similar to ``Log-concave part'', except for applying
 the induction to \eqref{eq:relaton:log-convex} and Lemma \ref{app:lem1}
 $(ii)$. 
\end{proof}

\noindent {\bf Acknowledgment:}
I am grateful to Tomasz Rolski for indicating the paper by Bender and Canfield
when we were working in a joint work. I also would like to thank Hidehiko Kamiya for useful
comments which are helpful in revising our paper.

\end{document}